\documentclass[reqno]{amsart}
\usepackage{amsmath,mathtools}
\usepackage{amsmath}
\usepackage{color}

\mathtoolsset{showonlyrefs}

\def\R{{\mathbb R}}

\newcommand{\C}{\mathbb C}

\newcommand{\Z}{\mathbb Z}
\newcommand{\T}{\mathbb T}
\newcommand{\e}{\mathrm e}

\newcommand{\rd}{{\rm d}}
\newcommand\I{\mathrm{i}}
\newcommand\jc[1]{\textcolor{black}{#1}}
\DeclareMathOperator{\Hphi}{Hess\,\phi}
\DeclareMathOperator{\rank}{rank}

\DeclareMathOperator{\diag}{diag}

\newtheorem{lemma}{Lemma}
\newtheorem{theorem}[lemma]{Theorem}
\newtheorem{corollary}[lemma]{Corollary}
\newtheorem{proposition}[lemma]{Proposition}
\theoremstyle{definition}
\newtheorem{definition}{Definition}
\theoremstyle{remark}
\newtheorem{remark}{Remark}
\usepackage{hyperref}

\begin{document}

\title[Sharp time decay estimates for the discrete Klein-Gordon equation]%
{Sharp time decay estimates for the discrete Klein-Gordon equation}
 \author[J.-C.\ Cuenin]{Jean-Claude Cuenin}
 \address[J.-C.\ Cuenin]{Department of Mathematical Sciences, Loughborough University, Loughborough,
 Leicestershire, LE11 3TU United Kingdom}
 \email{J.Cuenin@lboro.ac.uk}
 \author[I. A. Ikromov]{Isroil A.\ Ikromov}
\address[I. A. Ikromov]{Institute of Mathematics named after V.I. Romanovsky Academy of Sciences of Uzbekistan,
University Boulevard 15, 140104, Samarkand, Uzbekistan
}
\email{ikromov1@rambler.ru}
 
\date{July 21, 2021}

\begin{abstract}
We establish sharp time decay estimates for the the Klein-Gordon equation on the cubic lattice in dimensions $d=2,3,4$. The $\ell^1\to\ell^{\infty}$ dispersive decay rate is $|t|^{-3/4}$ for $d=2$, $|t|^{-7/6}$ for $d=3$ and $|t|^{-3/2}\log|t|$ for $d=4$. These decay rates are faster than conjectured by Kevrekidis and Stefanov (2005). The proof relies on oscillatory integral estimates and proceeds by a detailed analysis of the the singularities of the associated phase function. We also prove new Strichartz estimates and discuss applications to nonlinear PDEs and spectral theory. 
\end{abstract}

\maketitle

\section{Introduction and main results}

Dispersive estimates play a crucial role in the study of evolution equations. Proving such estimates often boils down to establishing decay estimates for the $\ell^{\infty}$ norm of the solution at time $t$ in terms of the $\ell^1$ norm of its initial data. It is by now well-established that the $\ell^1\to\ell^{\infty}$ decay estimates give rise to a whole family of mixed space-time norm estimates, called Strichartz estimates \cite{MR0512086,MR1151250,MR1646048}. For the continuous Klein-Gordon  equation such estimates have been established e.g.\ by Brenner \cite{MR744828}, Pecher \cite{MR795519} and Ginibre and Velo \cite{MR533218,MR533219}; let us also mention the textbook exposition by Nakanishi--Schlag \cite{MR2847755}. The dispersive estimates are halfway between those for the Schrödinger equation (for low frequencies) and the wave equation (for high frequencies), see \cite[2.5]{MR2847755}. 
In the discrete case the frequencies are bounded, and one might expect the same decay rate for the discrete Klein-Gordon equation (DKG) as for the discrete Schrödinger equation (DS). In fact, this was conjectured by Kevrekidis and Stefanov \cite{MR2150357}, who proved the sharp $|t|^{-d/3}$ decay rate for the DS in any dimension $d\geq 1$ and for the DKG in one dimension. The obstruction to proving similar estimates for DKG in higher dimensions was that, contrary to the DS, the fundamental solution does not separate variables. Apparently unbeknownst to the authors, earlier, Schultz \cite{MR1611132} had already made the striking observation that the decay rate for the discrete wave equation (to which the DKG reduces in the zero mass limit) in dimensions $d=2,3$ is $|t|^{-3/4}$ and $|t|^{-7/6}$, respectively, better than the conjectured estimates. The same decay rate for the $d=2$ DKG was established by Borovyk and Goldberg \cite{MR3650321}, who also proved that  the fundamental solution decays exponentially outside the light cone and that the decay rate is independent of additional parameters (mass and wave speeds in the coordinate directions).

\subsection{Discrete Klein--Gordon equation (DKG)}
Consider the Cauchy problem for the DKG:
\begin{equation}\label{Cauchy problem DKG}
\begin{cases}
&u_{tt}(t,x)-\Delta_x u(t,x)+u(t,x)=F(t,x),\quad (t,x)\in\R\times\Z^d,\\
&u(0,x)=f(x),\quad u_t(0,x)=g(x),
\end{cases}
\end{equation} 
where $\Delta_x$ is the discrete Laplacian,
\begin{align*}
\Delta_x u(t,x)=\sum_{|x-y|=1}u(y,t)-2du(t,x).
\end{align*} 
For the plane waves $e_{\xi}(x)=\e^{\I x\cdot\xi}$ we have $(1-\Delta_x)e_{\xi}(x)=\omega(\xi)^{\jc{2}}e_{\xi}(x)$, where $\omega:\T^d\to\R$ is the dispersion relation, given by
\jc{
\begin{align}\label{dispersion relation DKG}
\omega(\xi):=\sqrt{1+\sum_{j=1}^d 2(1-\cos(\xi_j))}.
\end{align}
}
We will identify the torus $\T^d=(\R/2 \pi Z)^d$ with the fundamental domain $[0,2\pi]^d$.
The solution to \eqref{Cauchy problem DKG} is given (for sufficiently regular data $f,g,F$) by Duhamel's formula,
\begin{equation}
\begin{split}\label{Duhamel formula}
u(t,x)&=\cos(t\sqrt{\mathbf{1}-\Delta_x})f(x)+\frac{\sin(t\sqrt{\mathbf{1}-\Delta_x})}{\sqrt{\mathbf{1}-\Delta_x}}g(x)\\
&+\int_0^t \frac{\sin((t-s)\sqrt{\mathbf{1}-\Delta_x})}{\sqrt{\mathbf{1}-\Delta_x}}F(s,x)\rd s.
\end{split}
\end{equation}
Here, $\cos(t\sqrt{\mathbf{1}-\Delta_x})e_{\xi}=\cos(t\omega(\xi))e_{\xi}$, $\frac{\sin(t\sqrt{\mathbf{1}-\Delta_x})}{\sqrt{\mathbf{1}-\Delta_x}}e_{\xi}=\frac{\sin(t\omega(\xi))}{\omega(\xi)}e_{\xi}$, and the action on Schwartz functions is defined by using the Fourier inversion formula. 

\subsection{Decay estimates}
The solution $u(t,x)$ is thus a sum of oscillatory integrals of the form
\begin{align}\label{oscillatory integrals KG}
I(t,x):=\int_{[0,2\pi]^d}\e^{\I (x\cdot \xi-t\omega(\xi))}a(\xi)\rd\xi,
\end{align}
where $a:[0,2\pi]^d\to \C$ is a smooth function; in fact, $a(\xi)=1$ or $a(\xi)=\omega(\xi)^{-1}$. We will be interested in obtaining time decay estimates on $I(t,x)$, uniformly in $x\in\Z^d$. In other words, we want to find (the largest possible) $\sigma>0$ such that 
\begin{align}\label{decay sigma}
\|I(t,\cdot)\|_{\ell^{\infty}(\Z^d)}\leq C(1+|t|)^{-\sigma},\quad t\in \R,
\end{align}
for some constant $C$ independent of $t$. The following theorem is our main result. 

\begin{theorem}\label{theorem decay KG}
For the oscillatory integrals \eqref{oscillatory integrals KG} the following estimates hold for all $t\in\R$:
\begin{align}
\|I(t,\cdot)\|_{\ell^{\infty}(\Z^2)}&\leq C(1+|t|)^{-\frac{3}{4}},\label{decay d=2}\\
\|I(t,\cdot)\|_{\ell^{\infty}(\Z^3)}&\leq C(1+|t|)^{-\frac{7}{6}},\label{decay d=3}\\
\|I(t,\cdot)\|_{\ell^{\infty}(\Z^4)}&\leq C(1+|t|)^{-\frac{3}{2}}\log(2+|t|),\label{decay d=4}
\end{align}
where $C$ is a constant independent of $t$. 
\end{theorem}

\begin{remark}
(i) The exponent in the estimates \eqref{decay d=2}, \eqref{decay d=3}, \eqref{decay d=4} is sharp. In fact, there are vectors $v\in \mathbb{R}^d$ and non-zero constants $c_d$ $(d=2,3,4)$ such that 
\begin{align}
\lim_{t\to +\infty} t^\frac34 |I(t, tv)|=c_2 \quad \mbox{for}\quad d=2,\\
\lim_{t\to +\infty} t^\frac76 |I(t, tv)|=c_3 \quad \mbox{for}\quad d=3,\\
\lim_{t\to +\infty} \frac{t^\frac32}{\log t} |I(t, tv)|=c_4 \quad \mbox{for}\quad d=4,
\end{align} 
see \cite{MR1330607} for the case $d=4$ and \cite{MR2854839} for similar observations.

\noindent (ii) We conjecture that for $d\geq 5$, the following estimate is true,
\begin{align}\label{decay d=5}
\|I(t,\cdot)\|_{\ell^{\infty}(\Z^d)}&\leq C_{a,d}(1+|t|)^{-\frac{2d+1}{6}}\log(2+|t|)^{d-4}.
\end{align}
We remark here that the original conjecture of Kevrekidis and Stefanov \cite{MR2150357} can be proved by relatively simple stationary phase arguments (see Remark \ref{remark d3case}). However, we expect that the better estimates \eqref{decay d=5} hold. 

\noindent (iii) The estimates \eqref{decay d=2}--\eqref{decay d=4} continue to hold if $\ell^{\infty}(\Z^d)$ is replaced by $L^{\infty}(\R^d)$. 

\noindent (iv) It will be necessary to prove Theorem \ref{theorem decay KG} in each case $d=2,3,4$ separately (see Propositions \ref{d2case}, \ref{d3case}, \ref{d4case}). This is in contrast to the Schrödinger case in \cite{MR2150357}, where the same proof works in all dimensions. 
\end{remark}

If $x=vt$, where $v\in \R^d$ is the velocity, then the oscillatory integral \eqref{oscillatory integrals KG} takes the form
\begin{align}\label{oscillatory integrals KG velocity}
J(t,v):=\int_{[0,2\pi]^d}\e^{\I t \Phi(v,\xi)}a(\xi)\rd\xi,\quad \Phi(v,\xi)=v\cdot\xi-\omega(\xi).
\end{align}
In fact, the proof of Theorem \ref{theorem decay KG} gives more precise information on the set of velocities for which the indicated decay occurs. These velocities are the images, under the map $\xi\mapsto\nabla\omega(\xi)$, of the critical points of the phase function $\Phi(\cdot,v)$. The time-decay of $J(t,v)$ is governed by the degeneracy of the phase function at the critical points lying within the support of $a$. For a fixed value of the parameter $v=v_0$ a critical point of $\Phi(\cdot)=\Phi(v_0,\cdot)$ is a point $\xi=\xi_0$ where the gradient $\nabla \Phi(\xi)$ vanishes; in the case of the phase function in \eqref{oscillatory integrals KG velocity} this happens if $\nabla\omega(\xi_0)=v_0$. If there are no critical points, then $J(t,v_0)$ decays faster than any polynomial in $t$. Since the Klein-Gordon equation has finite speed of propagation, $|\nabla\omega(\xi)|\leq 1$, there are no critical points if $|v_0|>1$; this follows from the alternative formula
\begin{align*}
\omega(\xi)^2=1+\sum_{j=1}^d\sin^2(\xi_j).
\end{align*}
The region $|v|\leq 1$ in velocity space is called the light cone. Hence, the solution $u(t,x)$ decays rapidly outside the light cone. In fact, since $\omega(\cdot)$ is analytic, the decay is exponential; this can be established by the method of steepest descent, like in \cite{MR1611132,MR3650321}, but we will not pursue the issue here. Inside the light cone, there are critical points. Generically (i.e.\ for most values of $v_0$) these critical points are non-degenerate, that is $\det\Hphi(\xi_0)\neq 0$. We note in passing that the Hessian is invariantly defined (i.e.\ invariant under changes of coordinates) at a critical point. The stationary phase method (see e.g.\ \cite{MR1232192}) yields a $|t|^{-d/2}$ decay at such points. However, unlike in the continuous case, there are caustics, i.e.\ regions inside the light cone $|v|<1$ where the solution decays slower, at a rate $|t|^{-\sigma}$ (possibly with an additional logarithmic loss $\log^k|t|$), where $d/2-\sigma>0$ is called the order of the caustic \cite{MR405513} or the singular index \cite{MR2919697}. In the simplest case, 
\begin{align}\label{corank one singularity}
\rank \Hphi(\xi_0)=d-1,\quad \nabla\det\Hphi(\xi_0)\neq 0.
\end{align}
Then $\Sigma:=\{\xi\in \T^d:\det\Hphi(\xi)=0\}$ is a smooth $d-1$ dimensional manifold, and there exists a (non-unique) kernel vector field, i.e.\ a smooth non-zero vector field $V$ along $\Sigma$ such that $\Hphi(\xi)(V(\xi))=0$ for all $\xi\in\Sigma$. A phase function $\Phi$ satisfying the condition \eqref{corank one singularity} is said to have a \emph{corank one} singularity. The simplest corank one singularity occurs when, for every $\xi\in\Sigma$, $\ker\Hphi(\xi)$ intersects $T_{\xi}\Sigma$ transversally, or equivalently, $(\nabla_{\xi}\det\Hphi)\cdot V\neq 0$ on $\Sigma$. This singularity is  called the (Whitney) ``fold". In the classification of Arnol'd \cite{MR0397777} the fold is called an $A_2$ singularity \cite{MR2896292}, and the oscillatory integral \eqref{oscillatory integrals KG velocity} reduces to an Airy integral in one variable \cite{MR196254,MR410050} (after integrating out the other $d-1$ variables by stationary phase). The next more complicated singularity, called the ``cusp" or $A_3$ singularity, occurs at points in $\Sigma$ where $(\nabla_{\xi}\det\Hphi)\cdot V$ vanishes to first order. A systematization of these ideas gives rise to the Thom-Boardman classes. We refer the interested reader to \cite{MR0494220,MR0341518,MR2896292} for an introduction to singularity theory. The situation becomes more complicated if we allow the parameter $v$ to vary. Then we have to consider a \emph{family} of functions as opposed to a single function. This introduces  (topological) notions of ``typicality" (or transversality) and ``stability" in the space of functions depending on a given number of parameters. In our case, the number of parameters equals the dimension $d$ of the underlying space $\Z^d$. Since we are considering $d\leq 4$ here, the only stable singularities are Thom's \emph{seven elementary catastrophes} \cite[15.1]{MR0494220}; in the terminology of Arnol'd these are the $A_2,A_3,A_4,A_5,D_4^-,D_4^+,D_5$ singularities. For the specific phase in \eqref{oscillatory integrals KG velocity}, we will show in Section \ref{section Singularities of the phase function} that all these, except $D_5$, appear in $d\leq 4$ dimensions. However, an additional (unstable) singularity appears in $d=4$. More precisely, in $d=2$, there are only $A_k$ ($k\leq 3$) singularities. In $d=3$, the phase function has a more degenerate $D_4^-$ type singularity. In $d=4$, the phase function has only critical points with finite multiplicity, and in the most degenerate case is similar to a hyperbolic singularity ($T_{4,4,4}$ in the classification of Arnol'd \cite[15.1]{MR2896292}). 
The critical point of a hyperbolic singularity is (complex) isolated;
however, our uniform estimates hold true for more general phase functions having non-isolated critical points. From the knowledge of the type of singularity, we can determine the singular index. For the $A_k$, $D_k$ singularities (and hence in $d=2,3$) this follows from the work of Duistermaat \cite{MR405513}.
For the unstable singularity in $d=4$ we use a result of Karpushkin \cite{MR731895}. 

For ease of reference, table \ref{table:normal forms} lists the normal forms and singularity indices for $A_k$, $D_k$ and $T_{p,q,r}$ type singularities (compare \cite[15.1]{MR2896292}, \cite[6.1.10]{MR2919697}). The normal form contains the ``active variables" only. By this we mean the following: The Splitting Lemma \cite[14.12]{MR0494220} says that, near a critical point, a smooth function $\phi:\R^d\to\R$ can be expressed in local coordinates as
\begin{align*}
\phi(x_1,\cdot,x_d)=f(x_1,\ldots,x_r)+Q(x_{r+1},\ldots,x_d),
\end{align*}
where $r$ is the corank of $\phi$ (or the number of active variables) at the critical point. Table \ref{table:normal forms} lists the normal forms of $f$, i.e.\ after a change of coordinates, $f$ reduces to one of the tabulated normal forms in the cases encountered here (at least in $d\leq 3$; the $T_{4,4,4}$ singularity is tabulated for comparison only).

\def\arraystretch{1.5}
\begin{table}\label{table:normal forms}
\begin{equation*}
\begin{array}{||l|l|l||}\hline\hline
  \text{Type}\ &  \text{Normal form}\  & \text{Singular index $d/2-\sigma$}\\ \hline
A_{k}, k\geq 1 & x_1^{k+1}& \frac{k-1}{2k+2}\\ \hline
  D_{k}, k\geq 4& x_1^2x_2+x_2^{k-1}&  \frac{k-2}{2k-2}\\ \hline
  T_{4,4,4}& x_1^4+x_2^4+x_3^4+ax_1x_2x_3,\, a\neq 0&  \frac{1}{2}\\
\hline\hline
\end{array}
\end{equation*}
\caption{Normal forms and singularity index for $A_k$, $D_k$ and $T_{4,4,4}$ type singularities.}
\end{table}

\subsection{Strichartz estimates}

As a consequence of Theorem \ref{theorem decay KG} we obtain Strichartz estimates. The proof follows from the (by now) standard argument of Keel--Tao \cite{MR1646048}. More precisely, we apply \cite[Theorem 1.2]{MR1646048} to the operator $U(t)=\e^{-\I t\sqrt{\mathbf{1}-\Delta_x}}$. For fixed $t$, $U(t)$ is a unitary operator on the Hilbert space $H=\ell^2_x$. Here and henceforth, $\ell_x^r=\ell^r(\Z^d)$ denote the spatial Lebesgue spaces. The mixed space-time Lebesgue spaces $L_t^q\ell_x^r$ are endowed with the norms
\begin{align*}
\|F\|_{L_t^q\ell_x^r}=\big(\int_{\R}\big(\sum_{x\in\Z^d}|F(t,x)|^r\big)^{\frac{q}{r}}\big)^{\frac{1}{q}}
\end{align*}
for $1\leq q,r<\infty$, with obvious modifications for $q=\infty$ or $r=\infty$. We recall that a pair of exponents $(q,r)$ is called $\sigma$-admissible \cite[Definition 1.1]{MR1646048} if 
\begin{align*}
q,r\geq 2,\quad (q,r,\sigma)\neq(2,\infty,1),\quad \frac{1}{q}+\frac{\sigma}{r}\leq\frac{\sigma}{2}.
\end{align*}
If equality holds in the last condition, then $(q,r)$ is said to be \emph{sharp} $\sigma$-admissible.
Taking $\sigma$ as the decay rate in \eqref{decay sigma}--\eqref{decay d=4}, the combination of Duhamel's formula \eqref{Duhamel formula} with Theorem \ref{theorem decay KG} and \cite[Theorem 1.2]{MR1646048} yields the following family of Strichartz estimates.

\begin{theorem}\label{theorem Strichartz}
Let $u$ be a solution of the Cauchy problem \eqref{Cauchy problem DKG}. Let $q,r,\overline{q},\overline{r}\geq 2$,
\begin{equation}\label{Strichartz (q,r)}
\begin{split}
\begin{cases}
\frac{1}{q}\leq \frac{3}{4}\left(\frac{1}{2}-\frac{1}{r}\right)\quad &\mbox{if } d=2,\\
\frac{1}{q}\leq \frac{7}{6}\left(\frac{1}{2}-\frac{1}{r}\right)\quad &\mbox{if } d=3,\\
\frac{1}{q}< \frac{3}{2}\left(\frac{1}{2}-\frac{1}{r}\right)\quad &\mbox{if } d=4,
\end{cases}
\end{split}
\end{equation}
and similarly for $\overline{q},\overline{r}$.
Then $u$ satisfies the estimate
\begin{align}\label{Strichartz estimate equation}
\|u\|_{L_t^q \ell_x^r}
\leq C_{q,r,\overline{q},\overline{r}}(\|f\|_{\ell_x^2}+\|g\|_{\ell_x^2}+
\|F\|_{L_t^{\overline{q}'}\ell_x^{\overline{r}'}}).
\end{align}
\end{theorem}

\begin{remark}\label{remark Strichartz pairs}
\noindent $(i)$ We call $(q,r)$ a Strichartz pair if it satisfies \eqref{Strichartz (q,r)}. For any such pair there exists $r_0\in [2,r]$ such that $(q,r_0)$ satisfies \eqref{Strichartz (q,r)} with equality (in $d=4$ we subtract a fixed, arbitrarily small $\epsilon>0$ from the right hand side). Thus there is a $\tau\in [0,1]$ such that
\begin{align*}
(1/q,1/r_0)=\tau(1/q_0,0)+(1-\tau)(0,1/2),
\end{align*}
where $q_0=8/3$ if $d=2$, $q_0=12/7$ if $d=3$ and $q_0=4/3+\epsilon'$ if $d=4$ (here $\epsilon'>0$ can be made arbitrarily small by choosing $\epsilon$ sufficiently small). Note that the Lebesgue spaces over $\Z^d$ form a filtration, i.e.\ $\ell^{p_1}(\Z^d)\subset \ell^{p_2}(\Z^d)$ for $p_1\leq p_2$. This fact, together with the Gagliardo-Nirenberg and Young's inequality (compare Remark 5 in \cite{MR2578796}) yields
\begin{align*}
\|u\|_{L_t^q\ell_x^r}\leq \|u\|_{L_t^q\ell_x^{r_0}}\leq \|u\|^{\tau}_{L_t^{q_0}\ell_x^{\infty}}\|u\|^{1-\tau}_{L_t^{\infty}\ell_x^{2}}=:\|u\|_{L_t^{q_0}\ell_x^{\infty}\cap L_t^{\infty}\ell_x^{2}}.
\end{align*}
Hence, all the Strichartz estimates in Theorem \ref{theorem Strichartz} can be subsumed in the inequality
\begin{align}\label{Strichartz subsumed}
\|u\|_{L_t^{q_0}\ell_x^{\infty}\cap L_t^{\infty}\ell_x^{2}}\leq 
C_{\overline{q},\overline{r}}(\|f\|_{\ell_x^2}+\|g\|_{\ell_x^2}+\inf
\|F\|_{L_t^{\overline{q}'}\ell_x^{\overline{r}'}}),
\end{align}
where the infimum is taken over all Strichartz pairs $(\overline{q},\overline{r})$.
\noindent $(ii)$ The Strichartz estimates are also commonly expressed in terms of mapping properties for the operator $U(t)=\e^{-\I t\sqrt{\mathbf{1}-\Delta_x}}$,
\begin{equation}\label{Strichartz mapping properties}
\begin{split}
\|U(t)f\|_{L_t^q\ell_x^r}&\leq C\|f\|_{\ell^2_x},\\
\big\|\int_{\R}U(-t)F(s)\big\|_{\ell_x^2}&\leq C\|F\|_{L_t^{\overline{q}'}\ell_x^{\overline{r}'}},\\
\big\|\int_{s<t}U(t-s)F(s)\big\|_{L_t^q\ell_x^r}&\leq C\|F\|_{L_t^{\overline{q}'}\ell_x^{\overline{r}'}}.
\end{split}
\end{equation}
Note that $(\mathbf{1}-\Delta_x)^{-1/2}$ is a bounded operator on $\ell_x^r$ for every $r\in [1,\infty]$, see \cite[Lemma~1]{MR2150357}; hence \eqref{Strichartz estimate equation} follows from \eqref{Strichartz mapping properties} and \eqref{Duhamel formula}.

\noindent $(iii)$ By \cite[Lemma 3.9]{Palle}
the following sharp $\sigma$-admissible estimates are best possible,
\begin{align*}
\begin{cases}
(q,r)=(\frac83,\infty)\quad &\mbox{if } d=2,\\
(q,r)=(\frac{12}{7},\infty)\quad &\mbox{if } d=3,\\
(q,r)=(\frac43+,\infty)\quad &\mbox{if } d=4,
\end{cases}
\end{align*}
in the sense that Strichartz estimates cannot hold for a pair $(\tilde{q},\infty)$ with $\tilde{q}<q$.
\end{remark}

\subsection{Discrete nonlinear Klein--Gordon equation}
Strichartz estimates can be used in conjunction with a contraction mapping argument to prove global well-posedness for certain nonlinear equations with small initial data. Here we consider the discrete nonlinear Klein--Gordon equation
\begin{align}\label{NLDKG}
u_{tt}-\Delta_x u+u\pm |u|^{2s}u=0,\quad \mbox{in }\R\times\Z^d,
\end{align}
where $(u(0),u_t(0))\in \ell_x^2\times \ell_x^2$ and $s$ satisfies
\begin{equation}\label{s for GWP}
\begin{split}
\begin{cases}
s\geq \frac{4}{3}\quad &\mbox{if } d=2,\\
s\geq \frac{6}{7}\quad &\mbox{if } d=3,\\
s> \frac{2}{3}\quad &\mbox{if } d=4.
\end{cases}
\end{split}
\end{equation}
Given the Strichartz estimates, the proof of the following theorem is standard (see e.g.\  \cite[Theorem 6]{MR2150357}), but we will give proofs of the PDE applications in Section \ref{Section proofs pde applications}.

\begin{theorem}[Global well-posedness for small data]\label{theorem Global well-posedness}
Assume that $s$ satisfies \eqref{s for GWP}. There exists $\epsilon>0$ and a constant $C$ so that, whenever $\|u(0)\|_{\ell_x^2}+\|u_t(0)\|_{\ell_x^2}\leq \epsilon$, then \eqref{NLDKG} has a unique global solution. Moreover, the solution satisfies
\begin{align*}
\|u\|_{L_t^q\ell_x^r}\leq C\epsilon
\end{align*}
for any Strichartz pair $(q,r)$.
\end{theorem}

Interpolating between the bounds of Theorem \ref{theorem decay KG} and energy conservation for the \emph{linear} DKG \eqref{Cauchy problem DKG} (without forcing term, i.e.\ $F=0$) we get the following decay estimates for the $\ell^p$-norm of the solution. For $2\leq p\leq \infty$,
\begin{align}\label{Lp decay}
\|u(t)\|_{\ell^p_x}\leq C_p(1+|t|)^{-\sigma(p-2)/p}\|(u(0),u_t(0))\|_{\ell^{p'}\times \ell^{p'}},
\end{align}
where $\sigma=\sigma_d$ is the decay rate for $p=\infty$, i.e.
\begin{align}\label{sigmad}
\sigma_d:=\begin{cases}
\frac{3}{4}\quad&\mbox{if }d=2,\\
\frac{7}{6}\quad&\mbox{if }d=3,\\
\frac{3}{2}-\quad&\mbox{if }d=4,
\end{cases}
\end{align}  
where $3/2-$ means $3/2-\epsilon$ for arbitrary fixed $\epsilon>0$. Before stating the next theorem we define
\begin{equation}\label{conditions on p_d and s_d}
\begin{split}
p_2:=\frac{1}{6}(13+\sqrt{97})\approx 3.8,\quad 
p_3:=\frac{1}{14}(27+\sqrt{337})\approx 3.2,\quad 
p_4:=3,\\
s_2:=\frac{1}{12}(1+\sqrt{97})\approx 0.9,\quad s_3:=\frac{1}{28}(\sqrt{337}-1)\approx 0.6,\quad s_4:=\frac{1}{2}.
\end{split}
\end{equation}

The next theorem is related to a conjecture by Weinstein \cite{MR1690199}. 

\begin{theorem}[Decay of small solutions]\label{theorem Decay of small solutions}
Let $d=2,3,4$ and $s>s_d$. Assume that $u$ satisfies the discrete nonlinear Klein--Gordon equation \eqref{NLDKG}. There exists $\epsilon>0$ such that, whenever $\|u(0)\|_{\ell^{p_d'}}+\|u_t(0)\|_{\ell^{p_d'}}\leq \epsilon$, then for every $p\in [2,p_d]$, we have
\begin{align}\label{Lp decay nonlinear}
\|u(t)\|_{\ell^p_x}\leq C_p(1+|t|)^{-\sigma(p-2)/p}\|(u(0),u_t(0))\|_{\ell^{p'}\times \ell^{p'}},
\end{align} 
where $\sigma=\sigma_d$. 
\end{theorem}

\begin{remark}
Theorem \ref{theorem Decay of small solutions} implies that no standing wave solutions $u(t,x)=\e^{\I\lambda t}\phi(x)$ are possible under the stated smallness assumption. Weinstein \cite{MR1690199} proved the existence of an excitation threshold for the nonlinear Schrödinger equation in the continuum and conjectured that, for $s\geq 2/d$, solutions with sufficiently small initial conditions satisfy $\lim_{t\to\infty}\|u(t)\|_{\ell^p_x}=0$ for all $p\in [1,\infty]$. On the lattice, Kevrekidis and Stefanov \cite{MR2150357} proved that, for $s>d/2$, suitably small solutions of the nonlinear Schrödinger equation actually decay like the free solution in $\ell^p_x$. They also obtained analogues for the nonlinear Klein--Gordon equation in one space dimension. Theorem \ref{theorem Decay of small solutions} establishes analogues of this result in dimensions $d=2,3,4$.
\end{remark}

\subsection{Resolvent estimates and spectral consequences}

Here we consider a stationary version of the discrete Klein--Gordon equation, namely
\begin{align}\label{stationary DKG}
\sqrt{\mathbf{1}-\Delta}\psi+V \psi=\lambda \psi,\quad \psi\in\ell^2(\Z^d).
\end{align} 
We start with a resolvent estimate for the unperturbed operator $\sqrt{\mathbf{1}-\Delta}$. The idea of using Strichartz estimates to prove resolvent estimates is not new. It has appeared e.g.\ in \cite{MR2150357,MR2140267,KochTataru2009,MR3615545,MR3841849}; the authors of \cite{MR3841849} attribute the argument to T.\ ~Duyckaerts. The following results are analogues to \cite[Proposition 3.3]{MR4009459} and \cite[Theorem 4]{MR2150357} for the stationary Schrödinger equation. Again, in the Klein--Gordon case, better estimates are possible; in particular, our resolvent estimates hold in $d=3$, whereas the estimates in \cite{MR4009459,MR2150357} require $d\geq 4$.

\begin{theorem}\label{prop. Resolvent estimates}
Let $d=3,4$. There exists a constant $C$ such that for all $\lambda\in \C$ we have the estimate
\begin{align*}
\|\psi\|_{\ell^{\frac{2\sigma}{\sigma-1}}}\leq C\|(\sqrt{\mathbf{1}-\Delta}-\lambda)\psi\|_{\ell^{\frac{2\sigma}{\sigma+1}}},
\end{align*}
where $\sigma=\sigma_d$ is given by \eqref{sigmad}.
\end{theorem}

\begin{corollary}\label{Corollary spectral consequences}
Let $d=3,4$. There exists $\epsilon>0$ such that, whenever $\|V\|_{\ell^{\sigma}}\leq \epsilon$, then the eigenvalue problem \eqref{stationary DKG} has no nontrivial solution.
\end{corollary}

\subsection{Organization of the paper}

Section \ref{section Singularities of the phase function} contains the main body of the paper. After some preliminary reductions, we perform a case-by-case study of the singularity structure of the phase function in dimensions $d=2,3,4$. This yields the proof of Theorem \ref{theorem decay KG}.
Section 3 contains proofs of the PDE applications. Numerical solutions to a system of equations that appears in the proof of the decay estimates are listed in an appendix.

\section{Singularities of the phase function}
\label{section Singularities of the phase function}

In this section we consider the oscillatory integrals \eqref{oscillatory integrals KG velocity}. To conform with standard notation we use (in this section only) the parameters $(\lambda,s)$ instead of $(t,v)$ and the integration variable $x$ instead of $\xi$. That is, we consider 
\begin{equation}\label{int}
J(\lambda, s)=\int_{[0, 2\pi]^d} e^{i\lambda \Phi(x,s)}dx,
\end{equation}
where
\begin{equation}
\Phi(x,s):= \omega(x)-sx,\quad \omega(x):=\sqrt{1+2d-2\sum_{j=1}^{d}\cos(x_j)}.
\end{equation}
Here and in the following $sx$ means $s\cdot x$.
Suppose $s=s^0$ is a fixed vector and $\Phi(x, s^0)$ has a critical point $x^0$ (perhaps non-unique, but we consider one of them). 

Note that $\omega(x)\geq 1$ for all $x\in [0,2\pi]$.
From now on we use the abbreviations $c_j:=\cos(x_j^0), s_j:=\sin(x_j^0)$ for $j=1,\dots, d$. 
Consider the functions 
\begin{align}
\phi_1(x)&:=\Phi(x, s^0)-\Phi(x^0, s^0),\label{phi1}\\
\phi_2(x)&:=\omega(x) +s^0 x+\Phi(x^0, s^0)\label{phi2}.
\end{align}
Obviously, $\phi_1(x^0)=0$ and $\nabla\phi_1(x^0)=0$. 
Moreover, $\phi_2(x^0)\neq 0$; otherwise, we would have $0=\phi_1(x^0)+\phi_2(x^0)=2\omega(x^0)\geq 2$. Instead of $\Phi$ we will consider the new function 
\begin{align}\label{phi}
\phi(x):=\phi_1(x)\phi_2(x)
\end{align}
and investigate the type of singularities of the critical point $x^0$. By the properties of $\phi_1$ and $\phi_2$ just mentioned, the singularity type (for the so-called weighted homogeneous cases) of the functions $\Phi$ and $\phi$ at $x^0$ is the same (and the critical value is zero for the latter). But, as we will see, singularities of $\phi$ are easier to study since we can avoid radicals.
\jc{The critical point equation yields $s^0=\nabla\omega(x^0)$. In order to avoid possible confusion between $s_j$ and the components of $s^0$ we will not use the latter notation any more.}

We introduce the new variables by $\eta=x-x^0$ and, by using \eqref{phi1}--\eqref{phi}, we have
\begin{equation}\label{phas}
\phi(\eta)=1+2d-2\sum_{j=1}^{d}\cos(x_j^0+\eta_j)-\frac{(\sum_{j=1}^{d}s_j\eta_j +1+2d-2\sum_{j=1}^{d}c_j)^2}{1+2d-2\sum_{j=1}^{d}c_j}.
\end{equation} 

For the Hessian matrix we have 
\begin{align}\label{Hessian}
(\Hphi(0))_{kj}:=\partial_k\partial_j\phi(0):=2\left(c_j\delta_{kj}-\frac{s_ks_j}{1+2d-2\sum_{l=1}^{d}c_l}\right),
\end{align}
where $\delta_{kj}$ is the Kronecker ``delta", i.e.\ $\delta_{kj}=1$ whenever $k=j$ and otherwise $\delta_{kj}=0$. 
Moreover, for the determinant of Hessian matrix the following relation 
\begin{align}\label{dethes}
\det \Hphi(0)=2^d\prod_{k=1}^{d}c_k\left(1-\sum_{j=1}^d \frac{s_j^2}{c_j(1+2d-2\sum_{l=1}^dc_l)}\right)
\end{align}
holds true, whenever $c_j\neq0, j=1,\dots d$. Note that $\det \Hphi(0)$ is a rational function of $\{c_j\}_{j=1}^d$ because $s_j^2=1-c_j^2$.
A standard limiting argument then yields 
\begin{align}\label{detheslim}
\det \Hphi(0)=2^d\left(\prod_{k=1}^{d}c_k-\sum_{j=1}^d \frac{s_j^2\prod_{m\neq j} c_m}{(1+2d-2\sum_{l=1}^dc_l)}\right)
\end{align}
for any $c_j$, not only for $c_j\neq0$. The following observation about \eqref{Hessian} will be useful later: $ \Hphi(0)$ is a rank one perturbation of a diagonal matrix.
By rank subadditivity we then have the following result.

\begin{lemma}\label{rank}
Let $D=2\diag(c_1,\ldots,c_d)$ and $S_{kj}=2s_ks_j/\omega(x^0)^{\jc{2}}$. Then
$\Hphi(0)=D-S$ and 
\begin{align}
\rank(D)-1\leq \rank(\Hphi(0))\leq \rank(D)+1.
\end{align}
\end{lemma}

In the following, we will consider each of the cases $d=2,3,4$ separately. In view of \eqref{rank} we distinguish each case into $d+1$ sub-cases according to how many of the $c_j$ are zero. 

We will need the following notion of homogeneity.
\begin{definition}
Let $\kappa=(\kappa_1,\ldots,\kappa_d)$ with $0<\kappa_j\leq 1$ for all $j=1,\ldots,d$. We say that a function $\phi:\R^d\to\R$ is homogeneous of degree $r\geq 0$ with respect to the weight $\kappa$ if for any $\lambda>0$,
\begin{align*}
\phi(\lambda^{\kappa_1}x_1,\ldots,\lambda^{\kappa_d}x_d)=\lambda^r \phi(x_1,\ldots,x_d).
\end{align*}
\end{definition}
For a given weight $\kappa$ and degree $r$ we will indicate terms of degree $>r$ by $\ldots$. If $\kappa$ and $r$ are clear from the context, we will not comment this further. For example, in \eqref{Taylor phi d=2} the meaning of ``$\ldots$" is the standard one ($\kappa_j=1$, $r=5$). We will usually normalize $\kappa$ such that $r=1$ (e.g.\ $\kappa_j=1/5$, $r=1$ in the previous example). 

If $\phi$ is analytic (as will be the case here), then it has an expansion in $\kappa$-homogeneous polynomials of increasing degrees. The polynomial with the lowest degree will be called the ``principal part", $\phi_{\rm pr}$.

\subsection{Two dimensions}

If $d=2$, then \eqref{Hessian} can be written as
$$
\det \Hphi(0)=4\left(c_1c_2-\frac{c_1s_2^2+c_2s_1^2}{5-2(c_1+c_2)}\right).
$$
\textbf{Case 1:} If $c_1=0, c_2=0$ then $s_j=\pm1$. Note that $s_j^3=s_j$, so we have 
\begin{equation}\label{Taylor phi d=2}
\phi(\eta)=-\frac15(s_1\eta_1+s_2\eta_2)^2-\frac2{3!}((s_1\eta_1)^3+(s_2\eta_2)^3)+\frac2{5!}((s_1\eta_1)^5+(s_2\eta_2)^5)+\dots.
\end{equation} 
We will show that the function $\phi$ has an $A_3$ type singularity at the point $\eta=0$.
Indeed, by using the change of variables $y_1= s_1\eta_1+s_2\eta_2, y_2=s_2\eta_2$ we can see that 
\begin{equation}
\phi(\eta(y))=-\frac15y_1^2-\frac2{3!}y_1y_2^2 +\dots,
\end{equation} 
where ``$\ldots$" is a sum of homogeneous polynomials of degree strictly bigger than $r=1$ with respect to the weight $\kappa=(1/2, 1/4)$. Hence, the principal part of the function has the form
\begin{equation}
\phi_{pr}(y)=-\frac15y_1^2-y_1y_2^2.
\end{equation} 
Changing variables again, $u_1=(5/2)^{1/2}y_2$, $u_2=\frac{1}{5}y_1+\frac{5}{2}y_2^2$, we can write this as
\begin{align*}
\phi_{pr}(u)=u_1^4-u_2^2.
\end{align*}
From Table \ref{table:normal forms} we infer (dropping the quadratic part, as we may) that this is an $A_3$ type singularity.

\textbf{Case 2:} Assume that exactly one of the $c_j$ is zero. Without loss of generality we assume that $c_1=0$ and $c_2\neq0$. Then since $s_1=\pm1$, it is easy to see that $\det \Hphi(0)\neq0$. 
Hence the function $\phi$ has a non-degenerate critical point at the origin, i.e.\ an $A_1$ type singularity. 
 
\textbf{Case 3:} If $c_1\neq0$ and $c_2\neq0$ and $\det \Hphi(0)=0$, then the condition $\det \Hphi(0)=0$ can be written as (using $s_j^2=1-c_j^2$)
\begin{align}\label{det Hphi=0 case 3 in d=2}
5-(c_1+c_2)-(\frac1{c_1}+\frac1{c_2})=0.
\end{align}  
\textbf{Case 3a:} If $c_1-\frac{s_1^2}{5-2(c_1+c_2)}=0$, then the condition $\det \Hphi(0)=0$ yields $s_2=0$, so $c_2=\pm 1$. We will argue that $\phi$ has an $A_2$ type singularity at the origin $\eta=0$. Indeed, we have
\begin{equation}
\phi(\eta)=c_2\eta_2^2-\frac2{3!}s_1\eta_1^3+ \dots,
\end{equation} 
where $s_1\neq0$ (otherwise we would have $c_1=\pm 1$ and hence $\det \Hphi(0)\neq0$), and $``\dots"$ means sum of homogeneous polynomials of degree $>1$ with respect to the weight $(1/3, 1/2)$. Consequently, the principal part of the function $\phi$ has the form
\begin{equation}
\phi_{pr}=c_2\eta_2^2-\frac2{3!}s_1\eta_1^3.
\end{equation} 
Therefore, the phase function has an $A_2$ type singularity.

\textbf{Case 3b:} The case $c_2-\frac{s_{2}^2}{5-2(c_1+c_2)}=0$ (and $\det \Hphi(0)=0$) is analogous.

\textbf{Case 3c:} Lastly, assume that $c_j\delta_{jk}-\frac{s_js_k}{5-2(c_1+c_2)}\neq0$ for $j,k=1,2$ (this covers all remaining cases). 
Then under the condition $\det \Hphi(0)=0$ (or equivalently \eqref{det Hphi=0 case 3 in d=2}) the function $\phi$ can be written as 
\begin{equation}
\phi(\eta)=\left(c_1-\frac{s_1^2}{5-2(c_1+c_2)}\right)\left(\eta_1-\frac{s_1s_2\eta_2}{5c_1-2c_1c_2-c_1^2-1}\right)^2-\frac{2}{3!}(s_1\eta_1^3+s_2\eta_2^3)+\dots
\end{equation}
We claim that $\phi$ has only an $A_k$ type singularity with $k\leq 2$. Indeed, a straightforward calculation shows that $\phi$ has $A_k$ type singularities with $k\ge3$ if and only if 
\begin{equation}\label{eq1}
\frac{s_1^4s_2^3}{(5c_1-2c_1c_2-c_1^2-1)^3}+s_2=0
\end{equation}
and 
\begin{equation}\label{eq2} 
5-(c_1+c_2)-(\frac1{c_1}+\frac1{c_2})=0.
\end{equation}
Since $s_2\neq0$, and $s_1\neq0$ the \eqref{eq1} under the condition \eqref{eq2} can be written as 
\begin{equation}
c_2^3s_1^4+c_1^3s_2^4=0
\end{equation}
or equivalently
\begin{equation}\label{eq3}
(1-c_1^2)^2c_2^3+(1-c_2^2)^2c_1^3=0.
\end{equation}
We claim that the system of equations \eqref{eq2} and \eqref{eq3} has no solution satisfying $|c_j|<1, j=1,2$. Indeed, if the pair $(c_1, c_2)$ is a solution to that system, then $c_1$ and $c_2$ have opposite signs. Without loss of generality, assume $1\ge c_1>0$ and $-1\le c_2<0$.

\begin{lemma}
Under the conditions $1\ge c_1>0$ and $-1\le c_2<0$, the system of equations \eqref{eq2}--\eqref{eq3} has no solutions. 
\end{lemma}

\begin{proof}  
Assume that \eqref{eq2} holds. The left hand side of \eqref{eq3} can be written as
\begin{align*}
\mbox{LHS of \eqref{eq3}}&=(1-c_1^2)^2c_2^3+(1-c_2^2)^2c_1^3
=c_1^3+c_2^3-2(c_1^2c_2^3+c_1^3c_2^2)+c_1^4c_2^3+c_1^3c_2^4\\
&=(c_1+c_2)(c_1^2-c_1c_2+c_2^2-2c_1^2c_2^2+c_1^3c_2^3)
\end{align*}
Since the signatures of $c_1$ and $c_2$ are different and since $0<|c_j|<1$, we have
\begin{align*}
c_1^2-c_1c_2+c_2^2-2c_1^2c_2^2+c_1^3c_2^3=
c_1^2(1-c_2^2)+c_2^2(1-c_1^2)-c_1c_2(1-c_1^2c_2^2)>0.
\end{align*}
Moreover, since $c_1+c_2\neq 0$ by \eqref{eq2}, it follows that \eqref{eq3} cannot hold.
\end{proof}

In summary, we have proved the following.

\begin{proposition}\label{d2case}
If $d=2$, then the phase function $\phi$ has only singularities of type $A_k$ with $k\le3$. More, precisely, If $c_1=c_2=0$ then $\phi$ has $A_3$ type singularity; if $c_1\neq0$ or $c_2\not=0$ then $\phi$ has $A_k$ type singularities with $k\le2$. For the corresponding oscillatory integral $J(\lambda, s)$ the estimate 
\begin{equation}
|J(\lambda, s)|\le C(1+|\lambda|)^{-\frac{3}{4}}
\end{equation}
holds, where the $C$ constant does not depend on $\lambda, s$.
\end{proposition}

\begin{remark}
The following comment appears in \cite{MR2150357} (page no. 1853) : ``Numerical simulations seem to confirm the validity
of (17) since in a two-dimensional numerical experiment for a DKG lattice, the best fit to the
decay was found to be $O(t^{-0.675})$. Note that (17) in \cite{MR2150357} is their conjectured bound $|J(\lambda, s)|\le C(1+|\lambda|)^{-\frac{d}{3}}$ (for $d=2$). Proposition \ref{d2case} actually shows that the decay rate is better than the numerical bound.
We speculate that this is related to the ``constant problem". As mentioned in the introduction, the bound of Proposition \ref{d2case} was already proved by Borovyk--Goldberg \cite{MR3650321}, without reference to the conjecture of Kevrekidis and Stefanov \cite{MR2150357}.
\end{remark}
 
\subsection{Three dimensions}
If $d=3$, then \eqref{Hessian} becomes
\begin{align}
\det \Hphi(0)=8c_1c_2c_3-\frac{8(c_1c_2s_3^2+c_1c_3s_2^2+c_2c_3s_1^2)}{1+2d-2(c_1+c_2+c_3)}.
\end{align}

\textbf{Case 1:} If $c_1=c_2=c_3=0$, then we have $s_j=\pm 1$ and 
\begin{equation}
\phi(\eta)=-\frac{(s_1\eta_1+s_2\eta_2+s_3\eta_3)^2}{7}-\frac{2}{3!}((s_1\eta_1)^3+(s_2\eta_2)^3+(s_3\eta_3)^3)+\dots
\end{equation}
Change of variables $y_1=s_1\eta_1+s_2\eta_2+s_3\eta_3, y_2=s_2\eta_2, y_3=s_3\eta_3$ yields
\begin{equation}
\phi(\eta(y))=-\frac{y_1^2}{7}-\frac{2}{3!}(y_1^3-3y_1^2(y_2+y_3)+3y_1(y_2+y_3)^2)-3y_2y_3(y_2+y_3))+\dots
\end{equation}
The principal part of $\phi$ with respect to the weight $(1/2, 1/3, 1/3)$ has the form
\begin{equation}
\phi_{pr}(y)=-\frac{y_1^2}{7}+y_2y_3(y_2+y_3).
\end{equation}
By a change of variables this can be reduced to 
$\phi_{pr}(z)=-z_1^2+z_2^2z_3-z_3^3$. Table~\ref{table:normal forms} tells us that this is a $D_4^-$ type singularity.

\textbf{Case 2:} Assume that exactly one of the $c_j$ is zero. Without loss of generality, we suppose that $c_1=0$ and $c_2\neq0, c_3\neq0$. Then 
\begin{align}
\det \Hphi(0)=-\frac{8 c_2c_3}{1+2d-2(c_2+c_3)}\neq0,
\end{align}
hence we have a non-degenerate critical point at $\eta=0$ (an $A_1$ singularity).

\textbf{Case 3:} Assume that exactly two of the $c_j$ are zero; without loss of generality, $c_1=0=c_2$ and $c_3\neq0$. Then $s_1=\pm1$ and $s_2=\pm1$. Hence, 
\begin{align}
\phi(\eta)&=c_3\eta_3^2-\frac{(s_1\eta_1+s_2\eta_2+s_3\eta_3)^2}{7-2c_3}-\frac{2}{3!}((s_1\eta_1)^3+(s_2\eta_2)^3+s_3\eta_3^3)-\frac2{4!}c_3\eta_3^4
\\&+\frac2{5!} ((s_1\eta_1)^5+(s_2\eta_2)^5+s_3\eta_3^5)+ \dots.
\end{align}
We change variables $y_1=s_1\eta_1+s_2\eta_2+s_3\eta_3, y_2=s_2\eta_2, y_3=\eta_3$ to find
\begin{align}
&\phi(\eta(y))=c_3y_3^2-\frac{y_1^2}{7-2c_3}-\frac{2}{3!}((y_1-y_2-s_3y_3)^3+y_2^3+s_3y_3^3)-\frac2{4!}c_3y_3^4\\
&+\frac2{5!} ((y_1-y_2-s_3y_3)^5+y_2^5+s_3y_3^5)+ \dots\\
&=c_3y_3^2-\frac{y_1^2}{7-2c_3}-\frac{2}{3!}((y_1-s_3y_3)^3-3(y_1-s_3y_3)^2y_2+3(y_1-s_3y_3)y_2^2+s_3y_3^3)\\
&-\frac2{4!}c_3y_3^4+\frac2{5!} ((y_1-y_2-s_3y_3)^5+y_2^5+s_3y_3^5)+ \dots.
\end{align}
The principal part with respect to the weight $(1/2, 1/4, 1/2)$ is given by
\begin{align}
\phi_{pr}(\eta(y))=c_3y_3^2-\frac{y_1^2}{7-2c_3}-\frac{2}{3!}3(y_1-s_3y_3)y_2^2=c_3y_3^2-\frac{y_1^2}{7-2c_3}-(y_1-s_3y_3)y_2^2.
\end{align}
We can rewrite $\phi(\eta(y))$ as
\begin{align}
\phi(\eta(y))&=c_3(y_3+\frac{s_3y_2^2}{2c_3})^2-\frac1{7-2c_3}(y_1+\frac{(7-2c_3)y_2^2}2)^2- \frac2{4!}c_3y_3^4\\
&-\frac{2}{3!}((y_1-s_3y_3)^3-3(y_1-s_3y_3)^2y_2+s_3y_3^3)\\
&- (\frac{s_3^2}{4c_3}- \frac{7-2c_3}4)y_2^4
+\frac2{5!} ((y_1-y_2-s_3y_3)^5+y_2^5+s_3y_3^5)+ \dots.
\end{align}
From this we see that if $s_3^2+2c_3^2-7c_3\equiv 1+c_3^2-7c_3\neq0$, then we have an $A_3$ type singularity at $y=\eta=0$. On the other hand, if 
$1+c_3^2-7c_3=0$ (i.e.\ $c_3=2/(7+3\sqrt{5})$), then we have
\begin{align}
\phi(\eta(y))&=c_3(y_3+\frac{s_3y_2^2}{2c_3})^2-\frac1{7-2c_3}(y_1+\frac{(7-2c_3)y_2^2}2)^2\\
&- \frac2{4!}c_3y_3^4
+\frac2{5!} (5(y_1-s_3y_3)y_2^4+\dots.
\end{align}
Note that the Taylor series of the function $\phi(\eta(y))$ has no term $cy_2^5$ with non-zero coefficient $c$.
By using the change of variables 
\begin{equation}
z_2=y_2, z_1=y_1+\frac{(7-2c_3)y_2^2}2, z_3=y_3+\frac{s_3y_2^2}{2c_3},
\end{equation}
we obtain (under the condition $1+c_3^2-7c_3=0$ we have $y_1-s_3y_3=z_1-s_3z_3$) 
\begin{align}
\phi(\eta(y(z)))=c_3z_3^2-\frac1{7-2c_3}z_1^2+\frac2{5!} (5(z_1-s_3z_3)z_2^4+\frac{s_3^4}{24c_3^3}z_2^6)+\dots,
\end{align}
where $``\dots" $ means sum of homogeneous polynomials of degree $>1$ with respect to the weight $(1/2, 1/6, 1/2)$. Note that the degree of the polynomial $(z_1-s_3z_3)z_2^4$ is $7/6>1$ with respect to that weight. Hence, the principal part is 
\begin{equation}
\phi_{pr}(z)=c_3z_3^2-\frac1{7-2c_3}z_1^2+\frac{s_3^4}{24c_3^3}z_2^6.
\end{equation} 
This is an $A_5$ type singularity at $z=y=\eta=0$, because if $1+c_3^2-7c_3=0$ (as we assumed), then $s_3\neq0$.

\textbf{Case 4:} Finally, assume $c_1\neq0, c_2\neq0, c_3\neq0$ and $\det \Hphi(0)=0$. This will be the longest and most delicate case. We claim that the function has an $A_k$ type singularity with $k\leq 3$. Indeed, due to Lemma \ref{rank}, the rank of the matrix $\Hphi(0)$ is at least $2$. Since we are assuming
$\det \Hphi(0)=0$, the rank of that Hessian matrix is exactly $2$.
 
\textbf{Case 4a:} Suppose that one of the $s_j$ vanishes. Without loss of generality, we will assume that $s_1=0$. Then $c_1=\pm 1$. Since $c_2\neq0, c_3\neq0$, arguing as in case 3a in $d=2$, we see that $\phi$ has an $A_2$ type singularity (we exclude the $A_1$ case since we assume $\det \Hphi(0)=0$).   

\textbf{Case 4b:} Further, assume $c_j\neq0, s_j\neq0, j=1,2,3$. 
Then for some $j$ we have $c_j-\frac{s_j^2}{7-2(c_1+c_2+c_3)}\neq0$; otherwise we would have 
\begin{equation}
det \Hphi(0)=\frac{8s_1^2s_2^2s_3^2}{(7-2(c_1+c_2+c_3))^3}\neq0,
\end{equation}
i.e.\ $\phi$ would have  an $A_1$ type (or non-degenerate) critical point at $\eta=0$. Without loss of generality, we will assume that $c_1-\frac{s_1^2}{7-2(c_1+c_2+c_3)}\neq0$.
The quadratic part of $\phi$ is given by
\begin{equation}
p_2(\eta):=  c_1\eta_1^2+c_2\eta_2^2+c_3\eta_3^2-\frac{(s_1\eta_1+s_2\eta_2+s_3\eta_3)^2}{7-2(c_1+c_2+c_3)}.
\end{equation}
For the kernel $\nabla p_2(\eta)=0$ we have the relation
\begin{equation}
\eta_1=\frac{s_1c_3}{c_1s_3}\eta_3, \quad  \eta_2=\frac{s_2c_3}{c_2s_3}\eta_3.
\end{equation}
\jc{Moreover, using the identity 
\begin{equation}
-\frac{(s_2\eta_2+s_3\eta_3)^2}{\frac{s_2^2}{c_2}+\frac{s_3^2}{c_3}}+c_2\eta_2^2+c_3\eta_3^2= 
\frac{c_2^2s_3^2\eta_2^2+c_3^2s_2^2\eta_3^2-2c_2c_3s_2s_3\eta_2\eta_3}{c_3s_2^2+c_2s_3^2},
\end{equation}
we find that
\begin{align}
&p_2(\eta)-\left(c_1-\frac{s_1^2}{7-2(c_1+c_2+c_3)}\right) \left(\eta_1 -\frac{s_1(s_2\eta_2+s_3\eta_3)}{c_1(7-2(c_1+c_2+c_3))-s_1^2}\right)^2\\ 
&=-\frac{(s_2\eta_2+s_3\eta_3)^2}{\frac{s_2^2}{c_2}+\frac{s_3^2}{c_3}}+c_2\eta_2^2+c_3\eta_3^2
=\frac{(c_2s_3\eta_2-c_3s_2\eta_3)^2}{c_3s_2^2+c_2s_3^2}.
\end{align}
}
Since $c_1-\frac{s_1^2}{7-2(c_1+c_2+c_3)}\neq0$ the function $\phi$ can be written as
\begin{equation}
\phi(\eta)=r_1(\eta_1+b_2\eta_2+b_3\eta_3)^2+r_2(\eta_2-\frac{s_2c_3}{c_2s_3}\eta_3)^2-\frac{2}{3!}(s_1\eta_1^3+s_2\eta_2^3+s_3\eta_3^3)+\dots,
\end{equation}
where $r_1, r_2, b_2, b_3$ are nonzero real numbers satisfying the conditions:  
\begin{equation}
r_1=c_1-\frac{s_1^2}{7-2(c_1+c_2+c_3)},\quad b_2\frac{s_2c_3}{c_2s_3}+b_3=-\frac{s_1c_3}{c_1s_3}.
\end{equation}
\textbf{Case 4bi:} Assume that 
\begin{equation}\label{case 4bi d=4}
7-\sum_{j=1}^3(c_j+1/c_j)=0\quad \mbox{and} \quad \sum_{j=1}^3\frac{s_j^4}{c_j^3}\neq0.
\end{equation}
We claim that $\phi$ has an $A_2$ type singularity at $\eta=0$.
Indeed, the first condition in \eqref{case 4bi d=4} is equivalent to $\det \Hphi(0)=0$. Under the second condition we can use change of variables 
\begin{equation}
y_1=\eta_1+b_2\eta_2+b_3\eta_3,\quad y_2=\eta_2-\frac{s_2c_3}{c_2s_3}\eta_3,\quad y_3=\eta_3
\end{equation}
to get
\begin{equation}
\phi(\eta)=r_1y_1^2+r_2y_2^2-\frac{2}{3!}\jc{\frac{c_3^3}{s_3^3}}(\sum_{j=1}^3\frac{s_j^4}{c_j^3})y_3^3+\dots,
\end{equation}
where $``\dots"$ means sum of homogeneous polynomials of degree $>1$ with respect to the weight $(1/2,1/2, 1/3)$. Hence, $\phi$ has an $A_2$ type singularity at $\eta=0$.    

\textbf{Case 4bii:} Assume 
\begin{equation}\label{case 4bii d=4}
7-\sum_{j=1}^3(c_j+1/c_j)=0\quad \mbox{and} \quad \sum_{j=1}^3\frac{s_j^4}{c_j^3}=0,
\end{equation}
and recall that the first condition simply means $\det \Hphi(0)=0$.
We then have the following relation,
\begin{equation}
7-2(c_1+c_2+c_3)-\frac{s_1^2}{c_1}=\left(\frac{s_2^2}{c_2}+\frac{s_3^2}{c_3}\right).
\end{equation}

Changing variables
\begin{equation}
y_1=\eta_1 -\frac{s_1(s_2\eta_2+s_3\eta_3)}{c_1(7-2(c_1+c_2+c_3))-s_1^2}, y_2=\eta_2- \frac{c_3s_2}{c_2s_3}\eta_3, y_3=s_2\eta_2+s_3\eta_3,
\end{equation}
we arrive at
\begin{align}
p_2(\eta(y))&=  \left(c_1-\frac{s_1^2}{7-2(c_1+c_2+c_3)}\right) y_1^2+ \frac{c_2^2s_3^2y_2^2}{c_3s_2^2+c_2s_3^2}\\
&=\frac{c_1}{7-2(c_1+c_2+c_3)}\left(\frac{s_2^2}{c_2}+\frac{s_3^2}{c_3} \right)y_1^2+\frac{c_2^2s_3^2y_2^2}{c_3s_2^2+c_2s_3^2}\\ &=\frac{c_1(c_3s_2^2+c_2s_3^2)}{c_2c_3(7-2(c_1+c_2+c_3))}y_1^2+\frac{c_2^2s_3^2y_2^2}{c_3s_2^2+c_2s_3^2}.
\end{align}
Under the conditions \eqref{case 4bii d=4} we can write
\begin{align}
\eta_1&=y_1 +\frac{s_1y_3}{c_1(7-2(c_1+c_2+c_3))-s_1^2},\\
\eta_2&=\frac{\frac{s_3^2y_2}{c_3}}{\frac{s_2^2}{c_2}+\frac{s_3^2}{c_3}}+\frac{\frac{s_2y_3}{c_2}}{\frac{s_2^2}{c_2}+\frac{s_3^2}{c_3}},\\
\eta_3&=-\frac{s_2y_2}{\frac{c_3}{s_3}(\frac{s_2^2}{c_2}+\frac{s_3^2}{c_3})}+\frac{y_3}{\frac{c_3}{s_3}(\frac{s_2^2}{c_2}+\frac{s_3^2}{c_3})},
\end{align}
or equivalently,
\begin{align}
\eta_1&=y_1 +\frac{\frac{c_2c_3s_1}{c_1}y_3}{c_3s_2^2+c_2s_3^2}, \\
\eta_2&=\frac{c_2s_3^2y_2}{c_3s_2^2+c_2s_3^2}+\frac{c_3s_2y_3}{c_3s_2^2+c_2s_3^2},\\
 \eta_3&=-\frac{s_3s_2c_2y_2}{c_3s_2^2+c_2s_3^2}+\frac{c_2s_3y_3}{c_3s_2^2+c_2s_3^2}.
\end{align}
A straightforward but tedious computation yields
\begin{align}
s_1\eta_1^3+s_2\eta_2^3+s_3\eta_3^3&=s_1\left(y_1 +\frac{\frac{c_2c_3s_1}{c_1}y_3}{c_3s_2^2+c_2s_3^2}\right)^3+
s_2\left( \frac{c_2s_3^2y_2}{c_3s_2^2+c_2s_3^2}+\frac{c_3s_2y_3}{c_3s_2^2+c_2s_3^2}\right)^3\\
&+s_3\left(-\frac{s_3s_2c_2y_2}{c_3s_2^2+c_2s_3^2}+\frac{c_2s_3y_3}{c_3s_2^2+c_2s_3^2}\right)^3\\
&=\frac{3s_1^3c_2^2c_3^2y_1y_3^2}{c_1^2(c_3s_2^2+c_2s_3^2)^2}+\frac{3y_2y_3^2c_2s_2s_3^2(c_3^2s_2^2-c_2^2s_3^2)}{(c_3s_2^2+c_2s_3^2)^3}+\dots \\
&=\frac{3s_1^3c_2^2c_3^2y_1y_3^2}{c_1^2(c_3s_2^2+c_2s_3^2)^2}+\frac{3y_2y_3^2c_2s_2s_3^2(c_3^2-c_2^2)}{(c_3s_2^2+c_2s_3^2)^3}+\dots,
\end{align}
where $``\dots"$ comprises a sum of terms with degree $>1$ with respect to the weight $(1/2, 1/2, 1/4)$. Under the condition 
$$\frac{s_1^4}{c_1^3} +\frac{s_2^4}{c_2^3}+\frac{s_3^4}{c_3^3}=0,$$ there is no term $cy_3^3$ with a non-zero coefficient $c$. 
Hence all other terms which are not indicated have degree $>1$ with respect to the weight $(1/2, 1/2, 1/4)$.
We claim that if the following condition (see \eqref{phipr case 4bii})
\begin{align}\label{condition 4bii d=4}
s_1^6c_2^5c_3^5(7-2(c_1+c_2+c_3))+ c_1^5s_3^2s_2^2(c_3^2-c_2^2)^2\neq0,
\end{align}
is satisfied, then the phase function has $A_3$ type singularity at $\eta=0$. 
Indeed, under the conditions \eqref{case 4bii d=4}, we have 
\begin{align}
\phi(\eta(y))&=  \left(\frac{c_1(c_3s_2^2+c_2s_3^2)}{c_2c_3(7-2(c_1+c_2+c_3))}\right) y_1^2+ \frac{c_2^2s_3^2y_2^2}{c_3s_2^2+c_2s_3^2}\\
&\jc{-\frac{s_1^3c_2^2c_3^2y_1y_3^2}{c_1^2(c_3s_2^2+c_2s_3^2)^2}-\frac{y_2y_3^2c_2s_2s_3^2(c_3^2-c_2^2)}{(c_3s_2^2+c_2s_3^2)^3}}+
\dots\\
&= \left(\frac{c_1(c_3s_2^2+c_2s_3^2)}{c_2c_3(7-2(c_1+c_2+c_3))}\right)\left( y_1\jc{-}\frac{s_1^3c_2^3c_3^3(7-2(c_1+c_2+c_3))y_3^2}{2c_1^3(c_3s_2^2+c_2s_3^2)^3} \right)^2\\
&+\frac{c_2^2s_3^2}{c_3s_2^2+c_2s_3^2}\left(y_2\jc{-}\frac{y_3^2s_2(c_3^2-c_2^2)}{2c_2(c_3s_2^2+c_2s_3^2)^2}\right)^2\\
&-\jc{\frac14}\frac{s_1^6c_2^5c_3^5(7-2(c_1+c_2+c_3))+ c_1^5s_3^2s_2^2(c_3^2-c_2^2)^2}{c_1^5(c_3s_2^2+c_2s_3^2)^5}y_3^4 +\dots,
\end{align}
where $``\dots"$ comprises a sum of terms with degree $>1$ with respect to the weight $(1/2, 1/2, 1/4)$. 
If we use change of variables 
\begin{align}
z_1=y_1\jc{-}\frac{s_1^3c_2^3c_3^3(7-2(c_1+c_2+c_3))y_3^2}{2c_1^3(c_3s_2^2+c_2s_3^2)^3},\quad 
z_2=y_2\jc{-}\frac{y_3^2s_2(c_3^2-c_2^2)}{2c_2(c_3s_2^2+c_2s_3^2)^2}, \quad z_3=y_3,
\end{align}
then we can write the last expression in the following equivalent form,
\begin{align}
\phi(\eta(y(z)))&=  \left(c_1-\frac{s_1^2}{7-2(c_1+c_2+c_3)}\right)z_1^2+\frac{c_2^2s_3^2z_2^2}{c_3s_2^2+c_2s_3^2}\\
&-\jc{\frac14}\frac{s_1^6c_2^5c_3^5(7-2(c_1+c_2+c_3))+ c_1^5s_3^2s_2^2(c_3^2-c_2^2)^2}{c_1^5(c_3s_2^2+c_2s_3^2)^5}z_3^4 +\dots,
\end{align}
where again $``\dots"$ denotes a sum of homogeneous polynomials of degree $>1$ with respect to the weight $(1/2,1/2,1/4)$. Actually, we have to show that, under the conditions \eqref{case 4bii d=4}, the terms of the Taylor expansion of order $\ge4$ do not contribute to the principal part 
\begin{equation}\label{phipr case 4bii}
\begin{split}
\phi_{pr}(\eta(y(z))):=  \left(c_1-\frac{s_1^2}{7-2(c_1+c_2+c_3)}\right)z_1^2+\frac{c_2^2s_3^2z_2^2}{c_3s_2^2+c_2s_3^2}-\\
-\jc{\frac14}\frac{s_1^6c_2^5c_3^5(7-2(c_1+c_2+c_3))+ c_1^5s_3^2s_2^2(c_3^2-c_2^2)^2}{c_1^5(c_3s_2^2+c_2s_3^2)^5}z_3^4.
\end{split}
\end{equation}
Let us consider the sum of monomials of degree $4$ in the original coordinate system,
\begin{align*}
c_1\eta_1^4+c_2\eta_2^4+c_3\eta_3^4&=c_1\left(y_1 +\frac{\frac{c_2c_3s_1}{c_1}y_3}{c_3s_2^2+c_2s_3^2}\right)^4+
c_2\left( \frac{c_2s_3^2y_2}{c_3s_2^2+c_2s_3^2}+\frac{c_3s_2y_3}{c_3s_2^2+c_2s_3^2}\right)^4\\
&+c_3\left(-\frac{s_3s_2c_2y_2}{c_3s_2^2+c_2s_3^2}+\frac{c_2s_3y_3}{c_3s_2^2+c_2s_3^2}\right)^4\\
&=\frac{c_2^4c_3^4}{(c_3s_2^2+c_2s_3^2)^4}
\underbrace{\sum_{j=1}^3\frac{s_j^4}{c_j^3}}_{=0}y_3^4+y_1P_1(y)+y_2P_2(y),
\end{align*}
where $P_1, P_2$ are homogeneous polynomials of degree $3$. The degree (with respect to the weight $(1/2, 1/2, 1/4)$) of any monomial of $y_jP_j(y)$, $j=1,2$, is at least $1/2+3/4=5/4$. Similarly one shows that higher order terms in the Taylor expansion also have degree at least $5/4$.

We now assume that the expression in \eqref{condition 4bii d=4} vanishes, i.e.
\begin{align} \label{condition 4bii d=4 violated}
s_1^6c_2^5c_3^5(7-2(c_1+c_2+c_3))+ c_1^5s_3^2s_2^2(c_3^2-c_2^2)^2=0.
\end{align}
If the function $\phi(\eta)$ had a singularity of type $A_k$ with $k>3$, then \eqref{condition 4bii d=4 violated} would have to hold.
However, it can be checked numerically that under the condition \eqref{case 4bii d=4}, the latter equation has no solution satisfying $|c_j|<1$ (see Appendix A).

We conclude that the phase function has $A_k$ $(k\le3)$ type singularities at $\eta=0$ whenever $c_j\neq0, j=1,2,3$.

In summary, we have the following result for $d=3$.

\begin{proposition}\label{d3case}
If $d=3$, then the phase function can have $A_k$ $(k\le 5)$ and $D_4$ type singularities. The corresponding oscillatory integral is estimated by
\begin{equation}
|J(\lambda, s)|\le C(1+|\lambda|)^{-\frac76},
\end{equation}
\end{proposition}
where $C$ does not depend on $\lambda, s$.

\begin{remark}\label{remark d3case}
\noindent $(i)$ Note that the singularity index for $A_5$ and $D_4$ is $\frac{1}{3}$. 

\noindent $(ii)$ For $d=2,3$ the most delicate case is $c_j\neq0$ for all $j$. Incidentally, the estimate $O(|t|^{-\frac{d}3})$, conjectured in \cite{MR2150357}, can obtained by simpler arguments, even in higher dimensions. Indeed, if $d\ge3$ and $c_j\neq0$ for all $j$ then the rank of Hessian matrix is at least $d-1$. Hence, by stationary phase, we get $J(\lambda, s)=\mathcal{O}(|\lambda|^{-(d-1)/2})$. Note that $(d-1)/2\ge d/3$ for $d\ge3$. Moreover if $d\ge4$ then $(d-1)/2\ge (2d+1)/6$.
For these estimates, conditions such as $|c_j|<1$ and $c_j^2+s_j^2=1$ are not needed; it is enough to assume $c_j^2+s_j^2\neq0$ $(j=1,\dots, d)$.
\end{remark}

\subsection{Four dimensions}
Let $d=4$. Again we first consider the most degenerate case.

\textbf{Case 1:} $c_1=c_2=c_3=c_4=0$. Then $s_j=\pm1$. 
Then we have 
\begin{align*}
\phi(\eta)&=-\frac{(s_1\eta_1+s_2\eta_2+s_3\eta_3+s_4\eta_4)^2}{9}-\frac{2}{3!}((s_1\eta_1)^3+(s_2\eta_2)^3+(s_3\eta_3)^3+(s_4\eta_4)^3)\\
&+\frac{2}{5!}((s_1\eta_1)^5+(s_2\eta_2)^5+(s_3\eta_3)^5+(s_4\eta_4)^5)+\dots
 \end{align*}
\jc{where $\ldots$ denotes higher order terms in the Taylor expansion that can be considered as a small perturbation of the principal part. We will show that this small perturbation can be removed by smooth change of variables, leading to a $T_{4,4,4}$ type singularity.}

\begin{lemma}\label{Tsing}
The function $\phi$ has $T_{4,4,4}$ type singularity at $\eta=0$.
\end{lemma}
\begin{proof}
We use the change of variables $y_4=s_1\eta_1+s_2\eta_2+s_3\eta_3+s_4\eta_4, y_j=s_j\eta_j$, for $j=1,2,3$. Then we have
\begin{equation}
\phi(\eta(y))=-\frac{y_4^2}9 -\sigma_1(y)\sigma_2(y)+\sigma_3(y)+\dots,
\end{equation}
 where $\sigma_j$ are elementary symmetric polynomials,
\begin{align}
\sigma_1(y):=y_1+\dots+ y_n,\\
\sigma_2(y):=y_1y_2+y_1y_3+\dots+ y_{n-1}y_n,\\
\sigma_3(y)=y_1y_2y_3+\dots+y_{n-2}y_{n-1}y_n,
\end{align}
and $n=d-1$ is the number of ``active variables" (so $n=3$ in the present case).
Here we used the relation
\begin{equation}
y_1^3+\dots+y_n^3-\sigma_1^3=3\sigma_3-3\sigma_1\sigma_2.
\end{equation}
If $\jc{n=3}$, then it can be shown that $\sigma_3-\sigma_1\sigma_2$ is linearly equivalent (by linear change of variables)
to $z_1z_2z_3$. 
Oscillatory integrals with such phase functions have been investigated by Karpushkin \cite{MR731895}.
They satisfy estimates of the form $\mathcal{O}(|\lambda|^{-\frac32} \log|\lambda|)$ as $|\lambda|\to\infty$.
It follows from the arguments of \cite{MR2854839}
that the same estimate holds when we add a small linear perturbation, i.e.\ the phase function is $z_1z_2z_3+O(|z|^4)$. In fact, Karpushkin's theorem holds for arbitrary small perturbations, but we only need this weaker result here. 
\jc{We explain this now more precisely:
The function $\phi(\eta(y))$ can be written as 
\begin{align*}
\phi(\eta(y))=b_1( y_1+y_2+y_3, y_4)(y_4-(y_1+y_2+y_3)^2\omega(y_1+y_2+y_3))^2+y_1y_2y_3+\dots,
\end{align*} 
where $b_1$ and $\omega $ are smooth functions of two and one variables respectively, with $b_1(0, 0)\neq0$ and $\omega(0)\neq0$. Hence, it is easy to see that the function
\begin{align*}
\phi_2(\eta(y))=\frac{\phi(\eta(y))}{\phi_1(\eta(y))}
\end{align*} 
 can be written as 
\begin{align*}
\phi_2(\eta(y))&=\tilde b_1( y_1, y_2, y_3,  y_4)(y_4-(y_1+y_2+y_3)^2\omega(y_1+y_2+y_3)+ \omega_1(y_1, y_2, y_3))^2\\&+y_1y_2y_3+\dots,
\end{align*}
where $\tilde b_1, \omega_1(y)$ are smooth functions satisfying  $b_1(0, 0, 0, 0)\neq0$ and $\omega_1$ satisfies the condition that for any nonnegative integer triples  $(k_1, k_2, k_3)$, with $k_1+k_2+k_3\le3$, one  has  $\partial_1^{k_1} \partial_2^{k_2}\partial_3^{k_3}\omega_1(0, 0, 0)=0$; moreover, ``$+\dots$" means $O(|y|^4)$ with $y=(y_1, y_2, y_3, y_4)$.
Then the total phase function can be written as 
$\Phi(y, s)=\phi_2(\eta(y))-sy$ up to an irrelevant additive constant. By using the stationary phase method in the $y_4$ variable, we obtain a perturbation of a function of three variables having the form $\phi_1(y_1, y_2, y_3):=y_1y_2y_3+\dots$. Thus, the total phase function $\Phi$ of three variables and four parameters can be written as
\begin{align*}
\Phi(y, s)&=\phi_1(y_1, y_2, y_3)-(s_1+g_1(s_4))y_1-(s_2+g_2(s_4))y_2-(s_3+g_3(s_4))y_3\\
&+ s_4((y_1+y_2+y_3)^2\omega(y_1+y_2+y_3)+
s_4\omega_2(y, s_4))+ s_4\omega_3(y, s_4),
\end{align*}
where $g_l$ $(l=1, 2, 3)$ are smooth functions, with $g_k(0)=0$, and   $\omega_j(y, s_4)$ $(j=2,3)$ are smooth functions such that for any nonnegative integer triples  $(k_1, k_2, k_3)$, with $k_1+k_2+k_3\le j$, we have  the relation $\partial_1^{k_1} \partial_2^{k_2}\partial_3^{k_3}\omega_j(0, 0, 0, 0)=0$.
Now we can apply the classical theorem of Karpushkin  \cite{MR731895} to the function $\Phi$ and obtain the required bound.}
\end{proof}

\textbf{Case 2:} Suppose that only one of the $c_j$ is zero; without loss of generality, $c_1=0 $ and $c_2\neq0, c_3\neq0, c_4\neq0$. Then by \eqref{dethes} we have 
$\det \Hphi(0)\neq0$. Thus, the phase function has an $A_1$ (non-degenerate) singularity. 

\textbf{Case 3:} Suppose that two of the $c_j$ are zero; without loss of generality, $c_1=0=c_2$ and $c_3\neq0, c_4\neq0$. \jc{Since $|s_1|=|s_2|=1$, the matrix $(S_{kj})_{kj=1}^2$ (see Lemma~\ref{rank} for the notation) has rank $1$. Hence the rank of the matrix $\Hphi(0)$ equals $3$.} Therefore, the corresponding integral decays as $\mathcal{O}(|\lambda|^{-3/2})$. This is sufficient for our result. 

\textbf{Case 4:} Suppose that three of the $c_j$ are zero; without loss of generality, $c_1=c_2=c_3=0$ and $c_4\neq0$. Then we have
\jc{
\begin{equation}
\phi(\eta)=c_4\eta_4^2-\frac{(\sum_{j=1}^4s_j\eta_j)^2}{9-2c_4}-\frac{2}{3!}((s_1\eta_1)^3+(s_2\eta_2)^3+
(s_3\eta_3)^3+s_4\eta_4^3)+\dots.
\end{equation}
By a linear change of variables $y_1=s_1\eta_1$,  $y_2=s_2\eta_2$, $y_3=\sum_{j=1}^4s_j\eta_j$, $y_4=\eta_4$,
\begin{align*}
\phi(\eta)&=c_4y_4^2- \frac{y_3^2}{\jc{9-2c_4}}\jc{-}\frac{2}{3!}(y_1^3+y_2^3-(y_1+y_2)^3) +\dots\\
&=c_4y_4^2- \frac{y_3^2}{\jc{9-2c_4}}\jc{-}\frac{2}{3!}(y_1^3+y_2^3-(y_1+y_2)^3)\\
&=c_4y_4^2- \frac{y_3^2}{\jc{9-2c_4}}+y_1y_2(y_1+y_2).
\end{align*}
where ``$\dots$" consists of sum of terms with degree $>1$ with respect to the weight $(1/3, 1/3, 1/2, 1/2)$, and
the principal part is given by the expression without ``$\dots$". In this case the phase function has a $D_4^{-}$ type singularity. The corresponding oscillatory integral decays as $\mathcal{O}(|\lambda|^{-5/3})$ (see Table~\ref{table:normal forms}).}

\textbf{Case 5:} Finally, we consider the case $c_1,c_2,c_3,c_4\neq0$. Then by Lemma \ref{rank} the rank of the quadratic part is at least $3$, and hence the corresponding phase function has $A_k$ type singularities. In this case we do not need to investigate the multiplicity of the corresponding phase function, because the corresponding oscillatory integral decays at least as $\mathcal{O}(|\lambda|^{-3/2})$, by stationary phase.
 
We record our findings for the $d=4$ case. 

\begin{proposition}\label{d4case}
If $d=4$, then the phase function has $\jc{T_{4, 4, 4}}, \, D_4^-$ or $ A_k$ type singularities.
The corresponding oscillatory integral is estimated by
\begin{equation}
|J(\lambda, s)|\le C(1+|\lambda|)^{-\frac32}\log(2+|\lambda|),
\end{equation}
where $C$ does not depend on $\lambda, s$.
\end{proposition}

\section{Proofs of the PDE applications}\label{Section proofs pde applications}

\begin{proof}[Proof of Theorem \ref{theorem Global well-posedness}]
We are going to use a contraction mapping argument (Banach's fix point theorem) in the metric space
\begin{align*}
X=\{u\in L_t^{\infty}\ell_x^{2}:\|u\|_{L_t^{q_0}\ell_x^{\infty}\cap L_t^{\infty}\ell_x^{2}}\leq 2C_{1,2}\|(f,g)\|_{\ell^2\times \ell^2}\},
\end{align*}
where $C_{1,2}$ denotes the constant in the Strichartz estimate \eqref{Strichartz subsumed} with $(\overline{q},\overline{r})=(1,2)$ and $q_0$ as in Remark \ref{remark Strichartz pairs}. For convenience, we also introduced the notation $(f,g):=(u(0),u_t(0))$ and
\begin{align*}
\|(f,g)\|_{\ell^2\times \ell^2}:=\|f\|_{\ell^2}+\|g\|_{\ell^2}.
\end{align*}
Consider the solution map
\begin{align*}
\Lambda u(t):=U_0(t)f+U_1(t)g\pm\int_0^tU_1(t-\tau)|u(\tau)|^{2s}u(\tau)\rd \tau,
\end{align*}
where 
\begin{align*}
U_0(t):=\cos(t\sqrt{\mathbf{1}-\Delta_x}),\quad U_1(t):=\frac{\sin(t\sqrt{\mathbf{1}-\Delta_x})}{\sqrt{\mathbf{1}-\Delta_x}}.
\end{align*}
A solution to the discrete nonlinear Klein--Gordon equation \eqref{NLDKG} is a fix point of $\Lambda$. 
To apply the contraction mapping argument, we first check that $\Lambda X\subset X$. Indeed, applying the Strichartz estimate \eqref{Strichartz subsumed} with $(\overline{q},\overline{r})=(1,2)$, we have, for $u\in X$,
\begin{align*}
\|\Lambda u\|_{X}
&\leq C_{1,2}(\|(f,g)\|_{\ell^2\times \ell^2}+\|u\|_{L_t^{2s+1}\ell_x^{2(2s+1)}}^{2s+1})\\
&\leq C_{1,2}(\|(f,g)\|_{\ell^2\times \ell^2}+\|u\|_{X}^{2s+1})\\
&\leq C_{1,2}(\|(f,g)\|_{\ell^2\times \ell^2}+(2C_{1,2}\epsilon)^{2s+1}).
\end{align*}
In the second inequality we used that the pair $(2s+1,2(2s+1))$ is a Strichartz pair and is thus controlled by the $X$-norm; in the last inequality we used the assumption on the initial data. For $\epsilon$ sufficiently small, the last expression is bounded by $2C_{1,2}\|(f,g)\|_{\ell^2\times \ell^2}$, and hence $\Lambda u\in X$. Similarly, using
\begin{align*}
\||u|^{2s}u-|v|^{2s}v\|_{L_t^1\ell_x^2}\leq C_s\|u-v\|_{L_t^{2s+1}\ell_x^{2(2s+1)}}(\|u\|^{2s}_{L_t^{2s+1}\ell_x^{2(2s+1)}}+\|v\|^{2s}_{L_t^{2s+1}\ell_x^{2(2s+1)}}),
\end{align*}
one verifies that $\Lambda:X\to X$ is a contraction.
\end{proof}

\begin{proof}[Proof of Theorem \ref{theorem Decay of small solutions}]
Here we are use the contraction mapping argument in the metric space
\begin{align*}
X=\{u \in L_t^{\infty}\ell_x^{2}:\|u(t)\|_{\ell_x^{p}}\leq 2C_p\langle t\rangle^{-\sigma_(p-2)/p}\|(u(0),u_t(0))\|_{\ell^{p'}\times \ell^{p'}}\},
\end{align*}
where $\langle t\rangle:=(1+|t|)$, $2\leq p\leq p_d$, $\sigma=\sigma_d$, and $C_{p}$ denotes the constant in the $\ell^p$ decay estimate for the linear equation \eqref{Lp decay}. By the latter, we have, for $u\in X$,
\begin{align*}
\|\Lambda u\|_{X}
&\leq C_{p}\langle t\rangle^{-\sigma(p-2)/p}\|(u(0),u_t(0))\|_{\ell^{p'}\times \ell^{p'}}\\
&+C_p\int_0^t\langle t-\tau\rangle^{-\sigma(p-2)/p}\|u(\tau)\|_{\ell_x^{(2s+1)p'}}^{2s+1}\rd \tau.
\end{align*}
It suffices to show that the second term is bounded by the first.
The assumptions on $p,s$ and \eqref{conditions on p_d and s_d} imply that
\begin{align}\label{condition for integral to converge}
(2s+1)p'\geq p,\quad \sigma(p-2)(2s+1)/p>1,
\end{align}
from which it follows that the second term is bounded by (using that $u\in X$)
\begin{align*}
C_p(2C_p\|(u(0),u_t(0))\|_{\ell^{p'}\times \ell^{p'}})^{2s+1}\int_0^{\infty} \langle t-\tau\rangle^{-\sigma(p-2)/p}\langle\tau\rangle^{-\sigma(p-2)(2s+1)/p}\rd \tau.
\end{align*}
By the second inequality in \eqref{condition for integral to converge}, the integral is convergent and bounded by a constant times $\langle t\rangle^{-\sigma(p-2)/p}$. Since $\|(u(0),u_t(0))\|_{\ell^{p'}\times \ell^{p'}}^{2s}\leq \epsilon^{2s}$, we may choose $\epsilon$ so small that $\Lambda u\in X$. The contractvity of $\Lambda$ again follows in a similar manner. Banach's fix point theorem then yields the existence of a unique solution, together with the a priori bound \eqref{Lp decay nonlinear}.
\end{proof}

\begin{proof}[Proof of Theorem \ref{prop. Resolvent estimates}]
We first remark that, by standard arguments, the Strichartz estimates \eqref{Strichartz mapping properties} can be localized in time to an interval $[0,T]$. We then apply the localized estimates with $(q,r)=(\overline{q},\overline{r})=(2,\frac{2\sigma}{\sigma-1})$ to the function $u(t):=\e^{\I t \lambda}\psi$, which satisfies the equation
\begin{align*}
\I\partial_t u+H_0 u =(H_0-\lambda)u,\quad u(0)=\psi,
\end{align*}
where $H_0:=\sqrt{\mathbf{1}-\Delta}$.
Here we assume without loss of generality that $\lambda$ is in the lower half plane; otherwise we consider $\e^{-\I t \lambda}\psi$.
The result is that
\begin{align*}
c(\lambda,T)\|\psi\|_{\ell^{\frac{2\sigma}{\sigma-1}}}\leq C(\|\psi\|_{\ell^2}+ c(\lambda,T)\|(H_0-z)\psi\|_{\ell^{\frac{2\sigma}{\sigma+1}}}),
\end{align*}
where $c(z,T):=\|\e^{\I t z}\|_{L_t^2(0,T)}\geq T^{1/2}$ by the assumption that $z$ is in the lower half plane. Dividing by $c(z,T)$ and letting $T\to\infty$ yields the claimed bound.
\end{proof}

\begin{proof}[Proof of Corollary \ref{Corollary spectral consequences}]
Since $(H_0-\lambda)\psi=-V\psi$, we have, by Hölder's inequality,

\begin{align*}
\|\psi\|_{\ell^{\frac{2\sigma}{\sigma-1}}}\leq C\|V\psi\|_{\ell^{\frac{2\sigma}{\sigma+1}}}
\leq C \|V\|_{\ell^{\sigma}} \|\psi\|_{\ell^{\frac{2\sigma}{\sigma-1}}}.
\end{align*}
If $\epsilon$ is so small that $C\epsilon<1$ we get a contradiction, unless $\psi=0$.
\end{proof}

\appendix

\section{Numerical results}

Here we list the numerical solutions of the system \eqref{case 4bii d=4}, \eqref{condition 4bii d=4 violated}. These were input in Wolfram alpha \cite{Wolfram} (access Oct 25, 2020) in the form
\begin{align*}
&7 - (x + y + z) - 1/x - 1/y - 1/z = 0,\\
&(1 - x^2)^2/x^3 + (1 - y^2)^2/y^3 + (1 - z^2)^2/z^3 = 0,\\
& (1 - x^2)^3 y^5 z^5 (7 - 2 (x + y + z)) + x^5 (1 - y^2) (1 - z^2) (z^2 - y^2)^2 = 0.
\end{align*}
Real solutions:
\begin{align*}
&x\approx-0.143939, y\approx 0.144912, z\approx 6.90076,\\
&x\approx -0.143939, y\approx 6.90076, z\approx 0.144912,\\
&x\approx 12.6977, y\approx-2.486, z\approx-0.402253,\\
&x\approx 12.6977, y\approx-0.402253, z\approx-2.486.
\end{align*}

\subsection*{Acknowledgements} The authors wish to thank the two anonymous referees for helpful remarks. We are garateful to Orif O.\ Ibrogimov for useful discussions.

\subsection*{Conflict of interest}

On behalf of all authors, the corresponding author states that there is no conflict of interest.

\bibliographystyle{plain}

\begin{thebibliography}{10}

\bibitem{MR2896292}
V.~I. Arnold, S.~M. Gusein-Zade, and A.~N. Varchenko.
\newblock {\em Singularities of differentiable maps. {V}olume 1}.
\newblock Modern Birkh\"{a}user Classics. Birkh\"{a}user/Springer, New York,
  2012.
\newblock Classification of critical points, caustics and wave fronts,
  Translated from the Russian by Ian Porteous based on a previous translation
  by Mark Reynolds, Reprint of the 1985 edition.

\bibitem{MR2919697}
V.~I. Arnold, S.~M. Gusein-Zade, and A.~N. Varchenko.
\newblock {\em Singularities of differentiable maps. {V}olume 2}.
\newblock Modern Birkh\"{a}user Classics. Birkh\"{a}user/Springer, New York,
  2012.
\newblock Monodromy and asymptotics of integrals, Translated from the Russian
  by Hugh Porteous and revised by the authors and James Montaldi, Reprint of
  the 1988 translation.

\bibitem{MR0397777}
V.~I. Arnol'd~d.
\newblock Remarks on the method of stationary phase and on the {C}oxeter
  numbers.
\newblock {\em Uspehi Mat. Nauk}, 28(5(173)):17--44, 1973.

\bibitem{MR3650321}
V. Borovyk and M. Goldberg.
\newblock The {K}lein-{G}ordon equation on {$\Bbb{Z}^2$} and the quantum
  harmonic lattice.
\newblock {\em J. Math. Pures Appl. (9)}, 107(6):667--696, 2017.

\bibitem{MR3841849}
J.-M. Bouclet and H. Mizutani.
\newblock Uniform resolvent and {S}trichartz estimates for {S}chr\"{o}dinger
  equations with critical singularities.
\newblock {\em Trans. Amer. Math. Soc.}, 370(10):7293--7333, 2018.

\bibitem{MR744828}
P. Brenner.
\newblock On space-time means and everywhere defined scattering operators for
  nonlinear {K}lein-{G}ordon equations.
\newblock {\em Math. Z.}, 186(3):383--391, 1984.

\bibitem{MR0494220}
Th. Br\"{o}cker.
\newblock {\em Differentiable germs and catastrophes}.
\newblock Cambridge University Press, Cambridge-New York-Melbourne, 1975.
\newblock Translated from the German, last chapter and bibliography by L.
  Lander, London Mathematical Society Lecture Note Series, No. 17.


\bibitem{MR3615545}
J.-C. Cuenin and C.~E. Kenig.
\newblock {$L^p$} resolvent estimates for magnetic {S}chr\"odinger operators
  with unbounded background fields.
\newblock {\em Comm. Partial Differential Equations}, 42(2):235--260, 2017.

\bibitem{MR405513}
J.~J. Duistermaat.
\newblock Oscillatory integrals, {L}agrange immersions and unfolding of
  singularities.
\newblock {\em Comm. Pure Appl. Math.}, 27:207--281, 1974.

\bibitem{MR533218}
J.~Ginibre and G.~Velo.
\newblock On a class of nonlinear {S}chr\"{o}dinger equations. {I}. {T}he
  {C}auchy problem, general case.
\newblock {\em J. Functional Analysis}, 32(1):1--32, 1979.

\bibitem{MR533219}
J.~Ginibre and G.~Velo.
\newblock On a class of nonlinear {S}chr\"{o}dinger equations. {II}.
  {S}cattering theory, general case.
\newblock {\em J. Functional Analysis}, 32(1):33--71, 1979.

\bibitem{MR1151250}
J.~Ginibre and G.~Velo.
\newblock Smoothing properties and retarded estimates for some dispersive
  evolution equations.
\newblock {\em Comm. Math. Phys.}, 144(1):163--188, 1992.

\bibitem{MR0341518}
M.~Golubitsky and V.~Guillemin.
\newblock {\em Stable mappings and their singularities}.
\newblock Springer-Verlag, New York-Heidelberg, 1973.
\newblock Graduate Texts in Mathematics, Vol. 14.

\bibitem{MR410050}
V.~Guillemin and D.~Schaeffer.
\newblock Remarks on a paper of {D}. {L}udwig.
\newblock {\em Bull. Amer. Math. Soc.}, 79:382--385, 1973.

\bibitem{MR2854839}
I.~A. Ikromov and D. M{\"u}ller.
\newblock Uniform estimates for the {F}ourier transform of surface carried
  measures in {$\Bbb R^3$} and an application to {F}ourier restriction.
\newblock {\em J. Fourier Anal. Appl.}, 17(6):1292--1332, 2011.

\bibitem{MR731895}
V.~N. Karpushkin.
\newblock Uniform estimates of oscillating integrals with a parabolic or a
  hyperbolic phase.
\newblock {\em Trudy Sem. Petrovsk.}, (9):3--39, 1983.

\bibitem{MR1330607}
V.~N. Karpushkin.
\newblock The leading term of the asymptotics of oscillatory integrals with a
  phase of the series {$T$}.
\newblock {\em Mat. Zametki}, 56(6):131--133, 1994.

\bibitem{MR1646048}
M.~Keel and T.~Tao.
\newblock Endpoint {S}trichartz estimates.
\newblock {\em Amer. J. Math.}, 120(5):955--980, 1998.

\bibitem{MR2150357}
P.~G. Kevrekidis and A. Stefanov.
\newblock Asymptotic behaviour of small solutions for the discrete nonlinear
  {S}chr\"{o}dinger and {K}lein-{G}ordon equations.
\newblock {\em Nonlinearity}, 18(4):1841--1857, 2005.

\bibitem{MR2578796}
P.~G. Kevrekidis, D.~E. Pelinovsky, and A.~Stefanov.
\newblock Asymptotic stability of small bound states in the discrete nonlinear
  {S}chr\"{o}dinger equation.
\newblock {\em SIAM J. Math. Anal.}, 41(5):2010--2030, 2009.

\bibitem{MR2140267}
H.~Koch and D.~Tataru.
\newblock {$L^p$} eigenfunction bounds for the {H}ermite operator.
\newblock {\em Duke Math. J.}, 128(2):369--392, 2005.

\bibitem{KochTataru2009}
H.~Koch and D.~Tataru.
\newblock Carleman estimates and unique continuation for second order parabolic
  equations with nonsmooth coefficients.
\newblock {\em Comm. Partial Differential Equations}, 34(4-6):305--366, 2009.

\bibitem{MR196254}
D. Ludwig.
\newblock Uniform asymptotic expansions at a caustic.
\newblock {\em Comm. Pure Appl. Math.}, 19:215--250, 1966.

\bibitem{MR2847755}
K. Nakanishi and W. Schlag.
\newblock {\em Invariant manifolds and dispersive {H}amiltonian evolution
  equations}.
\newblock Zurich Lectures in Advanced Mathematics. European Mathematical
  Society (EMS), Z\"urich, 2011.

\bibitem{Palle}
L. Palle. 
\newblock Mixed norm Strichartz-type estimates for hypersurfaces in three dimensions. 
\newblock Math. Z., https://doi.org/10.1007/s00209-020-02568-8, 2020. 

\bibitem{MR795519}
H. Pecher.
\newblock Low energy scattering for nonlinear {K}lein-{G}ordon equations.
\newblock {\em J. Funct. Anal.}, 63(1):101--122, 1985.

\bibitem{MR1611132}
P. Schultz.
\newblock The wave equation on the lattice in two and three dimensions.
\newblock {\em Comm. Pure Appl. Math.}, 51(6):663--695, 1998.


\bibitem{MR1232192}
E.~M. Stein.
\newblock {\em Harmonic analysis: real-variable methods, orthogonality, and
  oscillatory integrals}, volume~43 of {\em Princeton Mathematical Series}.
\newblock Princeton University Press, Princeton, NJ, 1993.
\newblock With the assistance of Timothy S. Murphy, Monographs in Harmonic
  Analysis, III.

\bibitem{MR0512086}
R.~S. Strichartz.
\newblock Restrictions of {F}ourier transforms to quadratic surfaces and decay
  of solutions of wave equations.
\newblock {\em Duke Math. J.}, 44(3):705--714, 1977.

\bibitem{MR4009459}
Y. Tadano and K. Taira.
\newblock Uniform bounds of discrete {B}irman-{S}chwinger operators.
\newblock {\em Trans. Amer. Math. Soc.}, 372(7):5243--5262, 2019.

\bibitem{MR1690199} 
M.~I. Weinstein.
\newblock Excitation thresholds for nonlinear localized modes on lattices. \newblock Nonlinearity, Vol. 12, 673-691, 1999.

\bibitem{Wolfram}
\newblock Wolfram Research, Inc., 
\newblock Wolfram|Alpha Knowledgebase, Champaign, IL (2018).

\end{thebibliography}

\end{document}